\documentclass[11pt,leqno]{amsart}
\usepackage{amssymb,amsmath,amsthm,amsfonts}
\usepackage{graphicx}
\usepackage{xcolor}   
\usepackage{caption}
\usepackage{subcaption}
\usepackage{enumerate}
\usepackage{float}

\numberwithin{equation}{section}

\newtheorem{thm}{Theorem}[section]
\newtheorem{lem}[thm]{Lemma}
\newtheorem{cor}[thm]{Corollary}

\newtheorem{rem}[thm]{Remark}

\newtheorem{definition}[thm]{Definition}

\newcommand{\R}{\mathbb{R}}
\newcommand{\C}{\mathbb{C}}

\newcommand{\N}{\mathbb{N}}
\newcommand{\Sh}{\mathcal{S}}

\newcommand{\dist}{\operatorname{dist}}

\newcommand{\ovdimB}{{\overline{\dim_B}}}
\newcommand{\lodimB}{{\underline{\dim_B}}}

\newcommand{\qad}{\dim_{\mathrm{qA}} }

\title{On concentric fractal spheres and spiral shells}

\author[Efstathios-K. Chrontsios-Garitsis]{Efstathios-K. Chrontsios-Garitsis}
\address{Department of Mathematics \\ University of Tennessee, Knoxville \\ 1403 Circle Dr \\ Knoxville, TN 37966}
\email{echronts@utk.edu, echronts@gmail.com}

\begin{document}

\maketitle

\begin{abstract}
We investigate dimension theoretic properties of concentric topological spheres, which are fractal sets emerging both in pure and applied mathematics. We calculate the box dimension and Assouad spectrum of such collections, and use them to prove that { fractal spheres cannot be shrunk through consecutive disjoint similar copies into a point at a polynomial rate}. We also apply these dimension estimates to quasiconformally classify certain spiral shells, a generalization of planar spirals in higher dimensions. This classification also provides a bi-H\"older map between shells and constitutes an addition to a general programme of research proposed by J. Fraser in \cite{Fra_spirals}.

\end{abstract}

\section{Introduction}

Concentric objects occur in various areas of pure and applied mathematics. In complex dynamics, the Fatou components of entire transcendental maps are typically of fractal boundary and converge to infinity (see for instance \cite{Baker, ComplexDyn}). Therefore, under a conformal inversion that keeps most geometric properties intact, the components of the map can be seen as fractals ``centered" at $0$. In the context of harmonic analysis, concentric Euclidean spheres are closely associated to weights of measures that lie in the important Muckenhoupt class \cite{Carlos}. In fact, polynomially-concentric sphere collections, i.e., unions of spheres $S_n$ centered at $0$ with radii equal to $n^{-p}$, for some $p>0$, provide one of the very few examples of bounded sets whose distance function lies in the Muckenhoupt class, for certain exponents. 

A set $S\subset \R^d$ is called a \textit{topological sphere}, if there is a homeomorphism $f:\R^d\rightarrow \R^d$ which maps the unit Euclidean sphere centered at $0$, denoted by $S(0,1)\subset \R^d$, onto $S$. A topological sphere can potentially be smooth, or extremely ``rough". 
A quantitative classification between the two cases is a difficult problem of many interpretations. As a result, it has received attention within numerous fields, such as topology, analysis on metric spaces and geometric group theory (see \cite{Bonk_QC} for a detailed survey on the topic). Despite such difficulties on the theoretical front, concentric topological spheres provide examples of concentric fractals that emerge in various applied areas, such as mathematical physics \cite{Concentric_phys}, \cite{concentric_phys2}, machine learning \cite{Concentric_machine},  mathematical models \cite{Conc_eq_med-appl}, and mathematical biology \cite{Tyson_concentric}.

It might seem natural that very rough fractals, which are disjoint and centered at a fixed point, would have to decrease a lot in size while they approach the center. More specifically, in concentric  topological sphere collections, where the spheres are allowed to be fractals of high fractal dimension, they need to converge to their center significantly fast in order to avoid overlaps. Hence, for any concentric collection of topological spheres, there are  certain restrictions on either their dimension, or on how fast they accumulate around their center. One of our main results is addressing this phenomenon for fractal spheres, by studying certain dimension notions for the whole collection. For what follows, we fix an integer $d\geq 2$, and we denote by $B(0,1)\subset \R^d$ the open unit ball in $\R^d$ centered at $0$. We formalize the notion of concentric sphere collections in the following definition.

\begin{figure}
    \centering
    \includegraphics[width=0.45\linewidth]{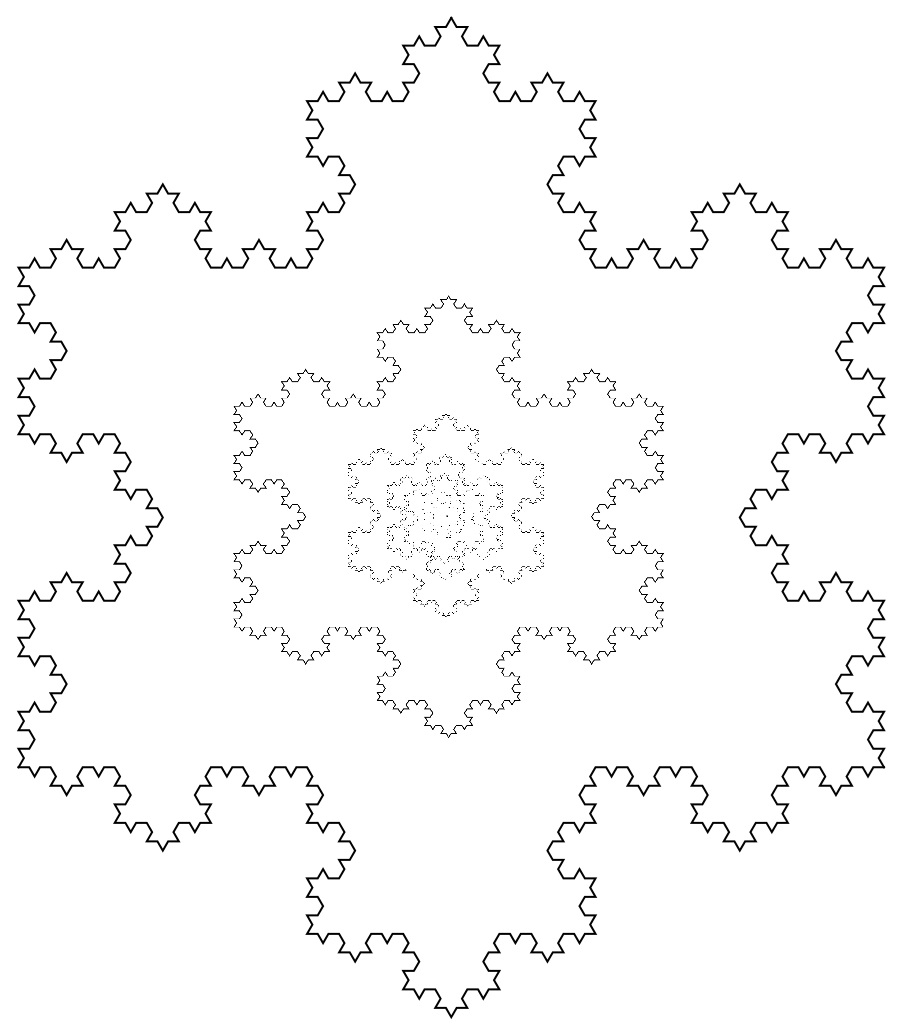}
    \caption{An example of a concentric collection of snowflakes.}
    \label{fig:snowflake}
\end{figure}

\begin{definition}\label{def: concentric}
    Let $S_0\subset \R^d$ be a topological sphere with $S_0= f(S(0,1))$, for some homeomorphism $f:\R^d \rightarrow \R^d$. For a sequence of similarities $\{g_n\}_{n\in\N}$, set $S_n:=g_n(S_0)$. We say that the union $S:=\bigcup_{n\in \N_0} S_n$ is a \textbf{concentric sphere collection} centered at $x_0$ and generated by $S_0$, {  for some $x_0\in f(B(0,1))$,} if there are $c_1, c_2\in (0,1]$, $n_0\in \N$, and a strictly decreasing sequence $\{a_n\}_{n\in \N}$ converging to $0$ such that:
    \begin{enumerate}[(i)]
        \item $d_H(S_n,\{x_0\})=c_1 a_n$, and
        \item $\dist(S_n,S_{n+1})=c_2 (a_n-a_{n+1}),$
    \end{enumerate} for all $n\geq n_0$.
    Moreover, if there is $p>0$ so that  $a_n=n^{-p}$, for all $n\geq n_0$, then we say that $S$ is a \textbf{$p$-(polynomially) concentric sphere collection} centered at $x_0$ and generated by $S_0$, and we denote it by $S_p=S_p(x_0,S_0)$.
\end{definition}

An example of a planar concentric sphere collection is depicted in Figure \ref{fig:snowflake}, where $S_0$ is the Koch snowflake (see \cite{falconer}). Note that in Definition \ref{def: concentric}, $\dist(S_n,S_{n+1})$ denotes the usual Euclidean distance between the sets $S_n,S_{n+1}$, and $d_H$ denotes the Hausdorff distance (see Section \ref{sec: Background} for the definition). The latter is a natural way to talk about convergence of a collection of sets to a point (see, for instance, \cite[Section 4.13]{Mat_book}, \cite[Chapter 5]{Fra_book}).
Therefore, one could perceive Definition \ref{def: concentric}  as the process of ``shrinking" the topological sphere $S_0$ into a point $x_0$, using contracting similarities $g_n$ of ratio approximately $a_n$.
Such problems have attracted a lot of interest from various areas of mathematics, and they have been extensively studied both from a deterministic \cite{Barany_dH_conv1, Barany_dH_conv2} and a probabilistic \cite{random_dH} point of view.
One of the main contributions of this paper is to emphasize that dimension-theoretic techniques can be crucial to the understanding of this phenomenon. Namely, we adopt such techniques in order to prove that a topological sphere $S_0$ of { non-trivial lower box dimension $\lodimB S_0$} (see Section \ref{sec: Background} for the definition) cannot give rise to polynomially concentric sphere collections of  any degree.

\begin{thm}\label{thm: no p-pol concentric sphere coll}
    Let $S_0\subset \R^d$ be a topological sphere with { $\lodimB S_0\in (d-1,d]$.} For any $p>0$, there is no $p$-concentric sphere collection generated by $S_0$.
\end{thm}

It should be noted that it is always possible to shrink a topological sphere $S_0$ into a point through a sequence of similarities as described in Definition \ref{def: concentric}, as long as the sequence $a_n$ converges fast enough to $0$. See Section \ref{sec: Final Remarks} for a detailed discussion. Theorem \ref{thm: no p-pol concentric sphere coll} establishes that sequences of the form $n^{-p}$ (or of slower rate of convergence to $0$) are simply not fast enough to shrink a topological sphere of non-trivial {  lower} box dimension. Hence, all concentric fractal spheres encountered in the aforementioned applications \cite{Concentric_phys, concentric_phys2, Concentric_machine, Conc_eq_med-appl} either ``smoothen" out as they approach the common center, or they shrink at an exponential rate. 
On the other hand, there are still various examples of topological spheres in $\R^d$ of box dimension equal to $d-1$ that generate polynomially concentric collections. {  For instance, the boundary of any convex regular polygon can be polynomially shrunk into its center of mass. }

The proof of Theorem \ref{thm: no p-pol concentric sphere coll} is achieved by assuming towards contradiction that for $d_0=\lodimB S_0>d-1$ and $p>0$, a $p$-concentric sphere collection $S_p(x_0,S_0)\subset \R^d$ exists. We then use dimension theoretic tools, and specifically the Assouad spectrum notion $\dim_A^\theta S_p$ introduced by Fraser-Yu \cite{Spectraa} (see Section \ref{sec: Background} for the definition), in order to show that such a collection would be of dimension larger than $d$,  surpassing the dimension of the ambient space $\R^d$ and leading to a contradiction. A partial version of Theorem \ref{thm: no p-pol concentric sphere coll} can be proved by providing lower bounds for the upper box dimension of such a collection, which  results in a contradiction only for certain values of $p$. Hence, the Assouad spectrum is  essential in achieving the result in full generality with this approach. Theorem \ref{thm: no p-pol concentric sphere coll} is the first application of the Assouad spectrum in this context of ``rigidity" of fractal collections.
Moreover, the property of the Assouad spectrum to trace ``finer" geometric data is also crucial for  ``classification" problems under families of mappings. We establish Assouad spectrum estimates for concentric sphere collections (see Section \ref{sec: dimensions}), which we apply to classify certain objects, up to classes of homeomorphisms.

An instance of a planar set resembling a collection of concentric circles is the spiral of the form
$$
\Sh(\phi):=\{ \phi(t) e^{it}: \, t>1 \},
$$ where $\phi:[1,\infty)\rightarrow(0,\infty)$ is a continuous, strictly decreasing function, with $\lim_{t\rightarrow \infty}\phi(t)=0$. Indeed,  such spirals share many geometric properties with  concentric collections of circles of the form $\cup_{n\in \N} S(0,\phi(n))$, while they hold a prominent role in fluid turbulence \cite{foi, moff, vass, vasshunt}, dynamical systems \cite{zub, spirals_ode}, and even certain types of models in mathematical biology \cite{Tyson_spirals, Shell_Mathbio}. Moreover, they provide examples of ``non-intuitive" fractal behavior (see \cite{dup}), while they have also been extensively studied due to their unexpected analytic properties. For instance, Katznelson-Nag-Sullivan \cite{unwindspirals} demonstrated the dual nature of such spirals, lying in-between smoothness and ``roughness", as well as their connection to certain Riemann mapping questions. The relation between the rate of convergence of $\phi(t)$ for $t\rightarrow \infty$ and the existence of Lipschitz and H\"older parametrizations has also been extensively studied by the aforementioned authors in \cite{unwindspirals}, by Fish-Paunescu in \cite{spirals}, and by Fraser in \cite{Fra_spirals}. In particular, Fraser in the latter paper focuses on spirals that resemble polynomially concentric circles, i.e., spirals where $\phi(t)=t^{-p}$, for $p>0$, and emphasizes further the ``fractal" nature of these objects, due to the peculiar behavior they exhibit in relation to certain dimension notions.

A natural higher dimensional generalization of spirals is the notion of spiral shells in $\R^3$. They exhibit similar behavior to spirals, which has attracted the interest of researchers in physics \cite{Shell_phys, Shell_engineer} and mathematical biology \cite{Tyson_spirals, Shell_Mathbio}. For $p>0$, we consider \textit{polynomial (spiral) shells} of the form
$$
\Sh_p:= \{ (u^{-p}\cos u \sin v, u^{-p}\sin u \sin v, u^{-p}\cos v)\in \R^3: \, u,\, v\in [1,\infty) \}\subset \R^3.
$$ The resemblance to the polynomial spirals mentioned above is perhaps already evident from the definition. In addition, spiral shells have a very similar geometric structure to $p$-polynomially concentric spheres in $\R^3$. See  Figure \ref{fig: shells}.

\begin{figure}[H]
\centering
\begin{subfigure}{.4\textwidth}
  \centering
  \includegraphics[width=\linewidth]{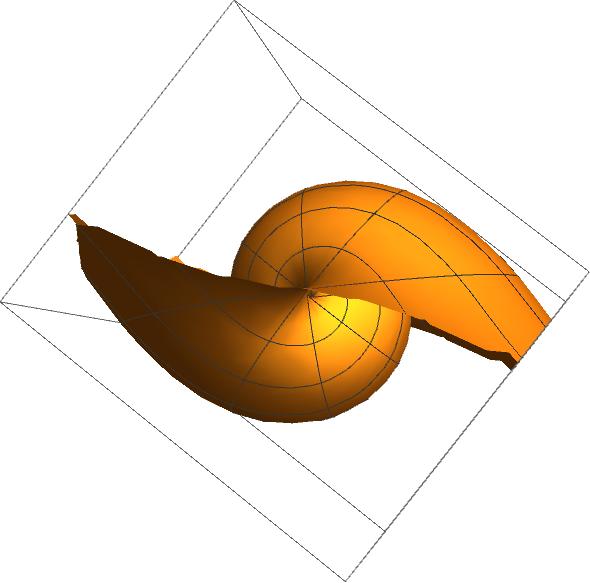}
\end{subfigure}
\begin{subfigure}{.4\textwidth}
  \centering
  \includegraphics[width=\linewidth]{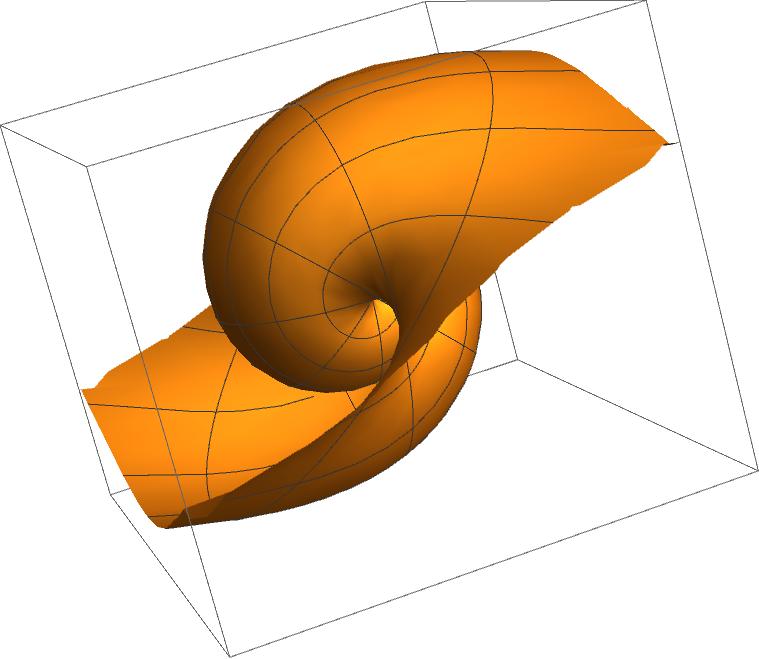}
\end{subfigure}
\caption{An example of a polynomial spiral shell for $p=1/2$.}
\label{fig: shells}
\end{figure}

One significant restriction in the analytic and geometric study of these objects is the lack of planar tools in $\R^3$. For instance, Liouville's theorem states that all conformal maps are M\"obius transformations in $\R^3$ (see for instance \cite[p.~15, 43]{Vais_book}). This was one of the motivations for the definition of \textit{quasiconformal maps} in $\R^d$, $d\geq 2$. More specifically, while conformal maps transform infinitesimal balls into infinitesimal balls, quasiconformal maps transform them into infinitesimal ellipsoids of globally bounded eccentricity (see Section \ref{sec: QC-Holder parametrization} for formal definition). The bound on said eccentricity is called \textit{dilatation}. If the dilatation of a quasiconformal map $f$ is at most some uniform constant $K\geq 1$, we say that $f$ is a \textit{$K$-quasiconformal map}. The theory of quasiconformal maps has a rich history, which still continues to develop actively, especially due to its many connections and applications within various fields of mathematics. We refer to the books by Lehto and Virtanen \cite{Leh_QC}, V\"ais\"al\"a \cite{Vais_book}, and Heinonen \cite{Heinonen} for thorough expositions on these maps and their importance in $\C$, $\R^d$, and metric spaces, respectively.

Distinguishing sets up to quasiconformal maps, particularly at the level of the dilatation, is in general a hard problem.
For instance, determining under what assumptions a topological sphere is a quasiconformal image of the usual unit sphere $S(0,1)$ is a long-standing open problem, with many implications in geometric group theory and topology (see \cite{Bonk_QC} for a survey on this topic and connections to Cannon's conjecture, see also \cite{Bonk_Meyer}). For planar topological spheres, i.e., simple closed curves homeomorphic to the unit circle, this classification problem was solved by Ahlfors in \cite{Ahlfors}, and their dimension-theoretic properties were fully determined by Smirnov in \cite{Smirnov}. In higher dimensions only partial progress has been made; see \cite{Bonk_Meyer, Meyer_Snowb} and the references therein.

Given the ties of spiral shells to several applications, as well as their relation to concentric topological spheres, it is natural to try and classify them under quasiconformal maps. Indeed, this is the outcome of our second main result. By applying the Assouad spectrum estimates we establish for concentric spheres (see Theorem \ref{thm: A-spec of S_p}), along with the dimension distortion properties of quasiconformal maps, proved by Tyson and the author in \cite{OurQCspec}, we achieve the desired classification of polynomial shells.

\begin{thm}\label{thm: main QC shells}
    For $p\geq q>0$, let $\Sh_p$ and $\Sh_q$ be two polynomial spiral shells. There is a $K$-quasiconformal mapping $F:\R^d\rightarrow \R^d$ with $F(S_p)=S_q$ if, and only if, $K\geq p/q$.
\end{thm}

Note that quasiconformal maps are also locally H\"older continuous, with locally H\"older continuous inverse (see Section \ref{sec: QC-Holder parametrization} for definitions). Therefore, in the proof of Theorem \ref{thm: main QC shells}, the quasiconformal map we construct, mapping $\Sh_p$ onto $\Sh_q$, is a bi-H\"older map between polynomial spiral shells, with the H\"older exponents quantitatively depending on the contracting ratios of the shells $p, q$. This contributes to a general direction proposed by J. Fraser in \cite{Fra_spirals}, on bi-H\"older classification of sets in $\R^d$. We discuss this connection further in Section \ref{sec: Final Remarks}.

This paper is organized as follows. Section \ref{sec: Background} includes the notation we follow, and all the relevant dimension notions and their properties. In Section \ref{sec: dimensions} we prove precise estimates for the upper box dimension and Assouad spectrum of $p$-concentric sphere collections, and apply similar dimension-theoretic arguments to prove Theorem \ref{thm: no p-pol concentric sphere coll}. In Section \ref{sec: QC-Holder parametrization} we recall the formal definition of quasiconformality, and state the relevant dimension distortion result we need for quasiconformal maps. Moreover, we demonstrate the shared dimension-theoretic properties between polynomially concentric spheres and spiral shells, which allows us to apply the quasiconformal dimension distortion result from \cite{OurQCspec} in order to prove Theorem \ref{thm: main QC shells}. Section \ref{sec: Final Remarks} contains further remarks and potential future directions that relate to this work.

\subsection*{Acknowledgements}
The  author wishes to thank Vasiliki Evdoridou, Manisha Garg and Carlos Mudarra for the interesting and motivating discussions on this work. Furthermore, the author thanks the anonymous referee for their valuable comments that significantly improved the manuscript.

\section{Background}\label{sec: Background}
We first establish the notation we follow.
Given two positive functions $h_1, h_2:(0,\infty)\rightarrow (0,\infty),$ we write $h_1(r) \lesssim h_2(r)$ if there is $C>0$, independent of $r$, such that $h_1(r)\leq C h_2(r)$, for all $r>0$. We call $C>0$ the \textit{comparability constant} of the relation $h_1(r) \lesssim h_2(r)$.
We write $h_1(r) \gtrsim h_2(r)$ if $h_2(r) \lesssim h_1(r)$. Lastly, we write $h_1(r) \simeq h_2(r)$ if both $h_1(r) \lesssim h_2(r)$ and $h_1(r) \gtrsim h_2(r)$ hold. Note that $h_1(r) \simeq h_2(r)$ in general implies that there are $C_1, C_2>0$, potentially distinct, such that $C_1 h_1(r)\leq h_2(r)\leq C_2 h_1(r)$ for all $r>0$. However, this implies that there is a single $C=\max\{ C_1^{-1}, C_2 \}>0$ such that $C^{-1} h_1(r)\leq h_2(r)\leq C h_1(r)$ for all $r>0$. In this case, we say that $C$ is the \textit{comparability constant} of the relation $h_1(r) \simeq h_2(r)$. 
Given $x\in \R^d$, $r>0$, we denote by $B(x,r)$ the open ball centered at $x$ of radius $r$, and by $S(x,r)=\partial B(x,r)$ the corresponding sphere in $\R^d$. Recall that a set $S\subset \R^d$ is called a \textit{topological sphere}, if there is a homeomorphism $f:\R^d\rightarrow\R^d$ with $S=f(S(0,1))$.
Given two compact sets $A,B\subset\R^d$, we denote their usual distance by 
{ $$
\dist(A,B)=\inf\{ |a-b|:\,a\in A,\, b\in B \},
$$ }
and their Hausdorff distance by
$$
d_H(A,B)=\max \Bigl\{ \sup\{\dist(\{a\},B):\, a\in A\},\, \sup\{\dist(A,\{b\}): \,b\in B\}\Bigr\}.
$$

We next define the dimension notions we need, along with their properties. We refer to the books by Falconer \cite{falconer} and Fraser \cite{Fra_book} for thorough expositions on these topics.
Let $E \subset \mathbb{R}^d$ be a non-empty bounded set. For $r>0$, we denote by $N_r (E)$ the smallest number of $r$-cubes needed to cover $E$, i.e., cubes that lie in the axes-oriented $r$-mesh of $\R^d$ and intersect $E$ (see \cite[p.~8]{Fra_book}).
The \emph{lower} and \emph{upper box dimensions} of $E$ are defined by
$$
\underline{\dim_B} E = \liminf_{ r \to 0} \, \frac{\log N_r (E)}{-\log  r}
\qquad
\text{and}
\qquad
\overline{\dim_{B}} E = \limsup_{ r  \to 0} \,  \frac{\log N_r (E)}{-\log  r},
$$
respectively.  If the lower and upper box dimensions are equal,  we call the common value the \emph{box dimension} of $E$, and denote it by $\dim_{B} E$. 

{ Note that if $\lodimB E=d_0^-$ and $\ovdimB E=d_0^+$, then for all $\epsilon>0$ there is $r_0=r_0(\epsilon)\in (0,1)$ such that
\begin{equation}\label{eq: N_r rel by dimB}
    r^{-d_0^-+\epsilon}\leq N_r(E)\leq r^{-d_0^+-\epsilon},
\end{equation} for all $r<r_0$. }

Recall that a map $g:\R^d\rightarrow\R^d$ is a \textit{similarity} if $g(x)=c_g \mathbf{A}x+t$ for all $x\in \R^d$, where $c_g>0$, $\mathbf{A}\in \mathcal{O}(\R,d)$ is a real orthogonal matrix, and $t\in \R^d$. We call the scalar $c_g$ the \textit{similarity ratio} of $g$. A standard property of the covering number $N_r(E)$ that we need is the way it is distorted by similarities. More specifically, if $g$ is a similarity map with similarity ratio $c$, then
$$
{  N_r(g(E))\simeq N_{r/c}(E)},
$$ for all $r\in (0,1)$, and some comparability constant that only depends on $d$ (see, for instance, \cite{falconer} Proposition 2.5). { If $\lodimB E=d_0^-$ and $\ovdimB E=d_0^+$, using the above and \eqref{eq: N_r rel by dimB} implies
\begin{equation}\label{eq: N_r similar}
    \left(\frac{r}{c}\right)^{-d_0^-+\epsilon}\lesssim N_r(g(E))\lesssim \left(\frac{r}{c}\right)^{-d_0^+-\epsilon},
\end{equation}} for all $r<\min\{c r_0, r_0\}$. This gives a relation between covering numbers of similar topological spheres that is crucial in our arguments in Section \ref{sec: dimensions}.

For an arbitrary (not necessarily bounded) set $E \subset \R^d$, the \textit{Assouad dimension} of $E$ is
$$
\dim_A E = \inf \left\{\alpha>0 \,:\, {\exists\,C>0\mbox{ s.t. } N_r(B(x,R) \cap E) \le C (R/r)^{\alpha} \atop \mbox{ for all $0<r\le R$ and all $x \in E$}} \right\}.
$$ While we do not use the Assouad dimension definition directly, we include it for the sake of completeness, and to relate it to the notion of the Assouad spectrum.

The Assouad spectrum, introduced by Fraser and Yu \cite{Spectraa}, is a one-parameter family of metrically defined dimensions which interpolates between the upper box dimension and the Assouad dimension. Specifically, the \textit{Assouad spectrum} of a set $E \subset \R^d$ { is a function $\theta \mapsto \dim_A^\theta E$}, where 
$$
\dim_A^\theta E=\inf \left\{\alpha>0 \,:\, {\exists\,C>0\mbox{ s.t. } N_r(B(x,r^\theta) \cap E) \le C (r^\theta/r)^{\alpha} \atop \mbox{ for all $0<r< 1$ and all $x \in E$}} \right\},
$$ for all $\theta\in (0,1)$. We may also refer to the value $\dim_A^\theta E$ as the \textit{$(\theta)$-Assouad spectrum} of E, for a specific  $\theta\in (0,1)$.

The Assouad spectrum essentially captures the growth rate of the covering number $N_r(B(x,R)\cap E)$ for scales $0<r\le R<1$ related by $R = r^\theta$. The map $\theta \mapsto \dim_A^\theta E$ is continuous at all $0<\theta<1$  \cite[Section 3.3.3]{Fra_book}, and for a fixed bounded set $E\subset \R^d$ we have
$$
\dim_A^\theta E \to \ovdimB E \quad \mbox{as $\theta \to 0$}, \qquad \dim_A^\theta E \to \dim_{qA}E \quad \mbox{as $\theta \to 1$}  ,
$$ 
where $\dim_{qA}E$ denotes the {\it quasi-Assouad dimension} of $E$, a variant of Assouad dimension introduced by L\"u and Xi \cite{quasi}. We always have $\dim_{qA} E \le \dim_A E$, and equality holds in many situations, including all sets of interest in this paper. We refer to \cite[Chapter 3]{Fra_book} for proofs of the above properties.

In Section \ref{sec: QC-Holder parametrization} we use a slightly modified version of the Assouad spectrum, where the relationship $R=r^\theta$ between the two scales of the Assouad dimension definition is relaxed to an inequality $R\ge r^\theta$. Namely, given $E\subset\R^d$ and $\theta\in (0,1)$, we set
$$
\dim_{A,reg}^\theta E:=\inf \left\{\alpha>0 \,:\, {\exists\,C>0\mbox{ s.t. } N_r(B(x,R) \cap E) \le C (R/r)^{\alpha} \atop \mbox{ for all $0<r\leq R^{1/\theta}<R<1$ and all $x \in E$}} \right\},
$$
This modification leads to the notion of {\it upper}, or {\it regularized Assouad spectrum}. See \cite{Spectraa} and \cite[Section 3.3.2]{Fra_book} for more information. Tyson and the author focused on this notion in \cite{OurQCspec}, as it aligns more naturally with properties of the Assouad dimension that are fundamental in dimension distortion results (see also \cite{HolomSpecChron}). For a fixed $\theta \in (0,1)$, the key relationship between the respective  values of the two spectra is the following (see  \cite[Theorem 3.3.6]{Fra_book}) 
\begin{equation}\label{eq:regularization}
{\dim_{A, reg}^\theta} E = \sup_{0<\theta'<\theta} \dim_A^{\theta'}E.
\end{equation} It is immediate from \eqref{eq:regularization} that if $\dim_A^\theta E$ is increasing in $\theta$, then $\dim_{A, reg}^\theta E=\dim_A^\theta E$. This is particularly useful, since $\dim_A^\theta E$ is often easier to calculate, while the dimension distortion results we use in Section \ref{sec: QC-Holder parametrization} are proved for the regularized Assouad spectrum.

Moreover, the relation between the different dimensions for a fixed bounded set $E$ follows by \eqref{eq:regularization} and the respective definitions. Namely,
\begin{equation} \label{eq: spec bounds}
\ovdimB E \leq \dim_A^\theta E\leq \dim_{A, reg}^\theta E \leq  \qad E \leq \dim_A E,
\end{equation} for all $\theta\in (0,1)$. We also state an upper bound on the Assouad spectrum involving the {  upper} box dimension, which reduces certain calculations in  Section \ref{sec: dimensions} (see \cite[Lemma 3.4.4]{Fra_book}).
\begin{lem}\label{le: dimA upper bd}
    {  Given a bounded non-empty set $E\subset \R^d$, we have
    $$
    \dim_A^\theta E\leq  \min\left\{ \frac{\ovdimB E}{1-\theta},\, d \right\},
    $$ for all $\theta\in (0,1)$.}
\end{lem}

\section{Dimensions and impossible concentric spheres} \label{sec: dimensions}

In this section, we first establish certain dimension-theoretic properties for polynomially concentric spheres. Namely, we estimate the  box dimension and Assouad spectrum of such collections. We then modify the arguments we employed while proving the aforementioned estimates, in order to show the restrictions of polynomially concentric sphere collections with regard to the generating sphere, which yields Theorem \ref{thm: no p-pol concentric sphere coll}.

Let $p>0$ and $S_0=f(S(0,1))\subset \R^d$ be a topological sphere, for some homeomorphism $f:\R^d\rightarrow \R^d$. Note that the topological dimension of $S(0,1)$ is equal to $d-1$, and, since the topological dimension is preserved under homeomorphisms, so is the topological dimension of $S_0$. Thus, we immediately have $\lodimB S_0\geq d-1$ (see, for instance, \cite{Heinonen} p.~62). Suppose that
$S_p=\cup_{n=0}^\infty S_n$ is a $p$-concentric sphere collection centered at some $x_0\in \R^d$ generated by $S_0$, as in Definition \ref{def: concentric}.

We can perform certain reductions on $S_p$. Without loss of generality, we can assume that $x_0=0$, since all the dimension notions we consider are translation-invariant. In addition, by finite stability of the box dimension and the Assouad spectrum (see \cite{Fra_book}, p.~18, 49), we can further assume that for the given $S_p$, we have $n_0=1$ in Definition \ref{def: concentric}. The above reductions imply that there is $c_1\in (0,1]$ such that
\begin{equation}\label{eq: d_H reduction}
    d_H(S_n, \{0\})=c_1 n^{-p},
\end{equation} for all $n\in\N$. In addition,  there is $c_2\in (0,1]$, independent of $n$, such that
\begin{equation}\label{eq: d(Sn,Sn+1) reduction}
    \dist(S_n, S_{n+1})=c_2(n^{-p}-(n+1)^{-p}),
\end{equation} for all $n\in \N$. Moreover, for any $n\in \N$, we have that
\begin{equation}\label{eq: dist(S_n,0)}
    \dist(S_n,\{0\})\simeq n^{-p},
\end{equation} where the comparability constant is independent of $n$. Indeed, the upper bound for \eqref{eq: dist(S_n,0)} follows trivially from \eqref{eq: d_H reduction}. In order to prove that $\dist(S_n,\{0\})\gtrsim n^{-p}$, let $z_n\in S_n$ be a point such that $|z_n|=\dist(S_n,\{0\})$. Since $S_j$ for $j\in\{n+1, n+2, \dots\}$ are all topological spheres with corresponding topological balls containing $0$, the line segment with end-points at $z_n$ and $0$ intersects each such $S_j$, say at $z_j\in S_j$. This implies that 
$$
|z_j-z_{j+1}|\geq \dist(S_j, S_{j+1})\simeq j^{-p}-(j+1)^{-p},
$$ for all $j\in\{n, n+1, \dots\}$. However, the sum of all of these distances $|z_j-z_{j+1}|$ has length less than the line segment where all these points lie on, namely
$$
\dist(S_n,\{0\})\gtrsim \sum_{j=n}^{\infty}(j^{-p}-(j+1)^{-p})=\sum_{j=1}^{\infty}(j^{-p}-(j+1)^{-p})-\sum_{j=1}^{n-1}(j^{-p}-(j+1)^{-p}).
$$ Using the fact that for any $m\in \N$ we have $\sum_{j=1}^{m}(j^{-p}-(j+1)^{-p})=1-(1+m)^{-p}$, the proof of \eqref{eq: dist(S_n,0)} is complete. Note that \eqref{eq: dist(S_n,0)} also implies that the similarity ratio of $g_n$ is in fact comparable to $n^{-p}$, for all $n\in \N$. This allows us to use \eqref{eq: N_r similar} in our arguments.

{ The above reductions are enough to start with the estimates on the box dimension of $S_p$ in the case where $d_0:=\lodimB S_0$ satisfies $d_0=d-1=\ovdimB S_0$.} While we could replace $d_0$ in the following proof by the explicit value $d-1$, we prefer to use the former notation to showcase how similar arguments are used for the proof of Theorem \ref{thm: no p-pol concentric sphere coll} for { $d_0=\lodimB S_0>d-1$}.

\begin{thm} \label{thm: dimB Sp}
For a $p$-concentric sphere collection $S_p$ as outlined above, with $d_0=\dim_B S_0=d-1$, we have
$$
\dim_B S_p =  \max\left\{ \frac{d}{1+p},\, d-1 \right\}.
$$
\end{thm}

\begin{proof}
Fix positive $d_1, d_2<d$ close to $d_0$ such that $d_1<d_0<d_2$. Due to $\dim_B S_0=d_0$ and \eqref{eq: N_r rel by dimB}, there is $r_0=r_0(d_1,d_2)>0$ such that
\begin{equation}\label{eq: dimB N_r with d1 d2}
    r^{-d_1}\leq N_r(S_0)\leq r^{-d_2},
\end{equation} for all $r\in (0,r_0)$. Moreover, since $f$ is a homeomorphism, the topological dimension of $B_0=f(B(0,1))$ is equal to $d$. It follows that $\dim_B B_0=d$. Hence, for any tiny $\epsilon>0$, there is $r_0'=r_0'(\epsilon)>0$ such that
\begin{equation}\label{eq: dimB of B_0}
    r^{-d+\epsilon}\leq N_r(B_0)\leq r^{-d-\epsilon},
\end{equation} for all $r\in (0,r_0')$.

Fix an arbitrary tiny $\epsilon>0$, and an arbitrary $r\in (0,\min\{r_0,r_0'\})$, so that \eqref{eq: dimB N_r with d1 d2}, \eqref{eq: dimB of B_0} hold, and let $n_r$  be the unique positive integer satisfying
\begin{equation}\label{eq: dimB choice of n_r}
n_r^{-p}-(n_r+1)^{-p} \leq \frac{r}{c_2d^{1/2}} < (n_r-1)^{-p}-n_r^{-p}.
\end{equation} In arguments that follow, we make a few more assumptions on how small $r$ is selected to be, which only involve uniform constants, such as $p$, $c_1$, $c_2$, $d$, and the uniform comparability constant for the similarity ratios of $g_j$ from \eqref{eq: dist(S_n,0)}. This ensures there is no loss of generality in these further assumptions on the size of $r$, while it eases the notation and reduces the number of technical inequalities.

By the choice of $n_r$ in \eqref{eq: dimB choice of n_r}, we can split $S_p$ into two disjoint sets, i.e.,  $S_p= S_p^+\cup S_p^-$, where
$$
S_p^+:=\bigcup_{n>n_r} S_n
$$  
$$
S_p^-:= \bigcup_{n \leq n_r} S_n,
$$
and consider $N_r(S_p^+)$ and $N_r(S_p^-)$ separately. The reason is that $r$-cubes that intersect $S_p^+$ are essentially disjoint with those that intersect $S_p^-$ due to \eqref{eq: d(Sn,Sn+1) reduction}  and \eqref{eq: dimB choice of n_r}, hence resulting in { $N_r(S_p)\simeq N_r(S_p^+)+N_r(S_p^-)$}. Moreover, by \eqref{eq: d(Sn,Sn+1) reduction} and the right-hand side of \eqref{eq: dimB choice of n_r}, the collection of $r$-cubes intersecting some $S_n\subset S_p^-$ for $n\leq n_r$ and the collection of $r$-cubes intersecting some $S_m$ for $n\neq m\leq n_r$ are essentially disjoint collections, implying that $N_r(S_p^-)=\sum_{n=1}^{n_r} N_r(S_n)$. Furthermore, by the left-hand side of \eqref{eq: dimB choice of n_r} and by \eqref{eq: d_H reduction}, \eqref{eq: dist(S_n,0)}, the scale $r$ is large enough for each cube in the $r$-mesh intersecting $B(0,c_1 n_r^{-p})$ to overlap with more than one of the spheres in $S^+$, resulting in $N_r(S^+)\simeq N_r(g_{n_r}(B_0))$, where the comparability constant is independent of $r$. Therefore, we have
\begin{equation}\label{eq: dimB N_r(Sp) split}
N_r\left( S_p \right)  \simeq  N_r(g_{n_r}(B_0)) + \sum_{j=1}^{n_r}  N_r\left( S_j \right)
\end{equation}
Note that $n_r$ increases as $r$ decreases. Hence, for small enough $r$ (in a way that depends only on $p, c_1, c_2, d$), the relation \eqref{eq: dimB choice of n_r} actually implies 
\begin{equation}\label{eq: dimB nr and r compar}
    n_r\simeq r^{-\frac{1}{p+1}},
\end{equation} where the comparability constant depends only on $p, c_1, c_2, d$. The similarity ratio of $g_{n_r}$ is comparable to $n_r^{-p}$. Since $n_r^p r\simeq n_r^{-1}$ by \eqref{eq: dimB nr and r compar}, without loss of generality, we can assume that $r$ is small enough, in a way solely dependent on $p$ and the comparability constant in the former relation, so that $n_r^p r$ is smaller than $\min\{r_0,r_0'\}$. This ensures that \eqref{eq: N_r similar} can be applied to relate $N_r(B_0)$ and $N_r(g_{n_r}(B_0))$ for the fixed $r$. By \eqref{eq: dimB of B_0}, \eqref{eq: dimB nr and r compar}, this results in
\begin{equation}\label{eq: dimB Nr(g_nr(B_0))}
    r^{-\frac{d-\epsilon}{1+p}}\simeq\left(\frac{1}{n_r^{p} r} \right)^{d-\epsilon} \lesssim N_r(g_{n_r}(B_0))\lesssim \left(\frac{1}{n_r^{p} r} \right)^{d+\epsilon} \simeq r^{-\frac{d+\epsilon}{1+p}},
\end{equation} for the fixed tiny $\epsilon>0$.

Since for any $j\in \{1,\dots, n_r\}$ the similarity ratio of $g_j$ is comparable to $j^{-p}$, and noting that $j^p r\leq n_r^p r<\min\{r_0,r_0'\}$ and $S_j=g_j(S_0)$, we similarly have by \eqref{eq: N_r similar}  that
\begin{equation}\label{eq: dimB Nr(S_j)}
    \left(\frac{r}{j^{-p}}\right)^{-d_1}\lesssim N_r(g_j(S_0))\lesssim \left(\frac{r}{j^{-p}}\right)^{-d_2}.
\end{equation}

Combining \eqref{eq: dimB N_r(Sp) split}, \eqref{eq: dimB Nr(g_nr(B_0))} and \eqref{eq: dimB Nr(S_j)} results in
\begin{equation}\label{eq: dimB double ineq Nr(S_p)}
    r^{-\frac{d-\epsilon}{1+p}}+\sum_{j=1}^{n_r}\frac{j^{-p d_1}}{r^{d_1}} \lesssim N_r(S_p) \lesssim r^{-\frac{d+\epsilon}{1+p}}+\sum_{j=1}^{n_r}\frac{j^{-p d_2}}{r^{d_2}}.
\end{equation}
Suppose that $pd_1<1$. Then, by standard partial sum estimates and \eqref{eq: dimB nr and r compar} we have
\begin{equation}\label{eq: dimB pd_1<1}
    \sum_{j=1}^{n_r}\frac{j^{-p d_1}}{r^{d_1}}\simeq n_r^{1-pd_1} r^{-d_1}\simeq r^{-\frac{1-pd_1}{p+1}-d_1}=r^{\frac{-1-d_1}{p+1}}.
\end{equation} Similarly, if $pd_1>1$, we have
\begin{equation}\label{eq: dimB pd_1>1}
    \sum_{j=1}^{n_r}\frac{j^{-p d_1}}{r^{d_1}}\simeq \frac{1}{pd_1-1} r^{-d_1}\simeq r^{-d_1}.
\end{equation}

A similar case study on the product $pd_2$ shows that for $pd_2<1$ we have
\begin{equation}\label{eq: dimB pd_2<1}
    \sum_{j=1}^{n_r}\frac{j^{-p d_2}}{r^{d_2}}\simeq n_r^{1-pd_2} r^{-d_2}\simeq r^{-\frac{1-pd_2}{p+1}-d_2}=r^{\frac{-1-d_2}{p+1}},
\end{equation} and for $pd_2>1$ we have
\begin{equation}\label{eq: dimB pd_2>1}
    \sum_{j=1}^{n_r}\frac{j^{-p d_2}}{r^{d_2}}\simeq \frac{1}{pd_2-1} r^{-d_2}\simeq r^{-d_2}.
\end{equation}

We now determine $\dim_B S_p$ based on the value of the product $pd_0$. Suppose $pd_0<1$, then $d_2$ can be chosen close enough to $d_0$ to ensure $pd_1<pd_2<1$. Thus, by \eqref{eq: dimB double ineq Nr(S_p)}, \eqref{eq: dimB pd_1<1}, \eqref{eq: dimB pd_2<1} we have
$$
r^{-\frac{d-\epsilon}{1+p}}+ r^{\frac{-1-d_1}{p+1}}\lesssim N_r(S_p)\lesssim r^{-\frac{d+\epsilon}{1+p}}+ r^{\frac{-1-d_2}{p+1}}.
$$ Note that $d_2> d-1$,  so $\epsilon$ can be chosen small enough so that $d_2+1> d+\epsilon$. Since $r$ was arbitrary, the above relation implies that
$$
    \frac{1+d_1}{p+1}\leq \dim_B S_p \leq \frac{1+d_2}{p+1}.
$$ Recall that $d_1$, $d_2$ are arbitrary and as close to $d_0$ as necessary, which combined with the above inequality is enough to show that in the case $pd_0<1$, we have
$$
\dim_B S_p =\frac{1+d_0}{p+1}=\frac{d}{p+1}.
$$

On the other hand, if $pd_0>1$, we similarly make sure $pd_2>pd_1>1$ by choice of $d_1, d_2$. Note that by $d_2>d-1$, the value $\epsilon$ can be chosen small enough so that $pd_2>1$ implies that $d+\epsilon<pd_2+d_2$, which gives
$$
\frac{d+\epsilon}{p+1}<d_2.
$$
Using a similar argument to that of the previous case, and the above inequality, along with \eqref{eq: dimB double ineq Nr(S_p)}, \eqref{eq: dimB pd_1>1}, \eqref{eq: dimB pd_2>1}, we can show that
$$
d_1\leq \dim_B S_p \leq d_2,
$$ which is enough to conclude that $\dim_B S_p=d_0$ as needed.

The case $p d_0=1$ is treated similarly, using \eqref{eq: dimB double ineq Nr(S_p)}, \eqref{eq: dimB pd_1<1}, \eqref{eq: dimB pd_2>1}, resulting in $\dim_B S_p=d_0=d-1$, due to the fact that
$$
\frac{1+d_0}{1+p}= \frac{1+d_0}{1+1/d_0}=d_0.
$$ All cases have been addressed and the proof is complete.

\end{proof}

\begin{rem}
	It might be tempting to consider $r\simeq n_r^{-p}$, so that all but finitely many of the spheres $S_j\subset S_p$ are covered by the $2^d$ $r$-cubes with vertices at $0$. However, this choice of $n_r$ would hinder any following covering arguments, and especially the crucial equation $N_r(S_p^-)=\sum_{j=1}^{n_r} N_r(S_j)$ would no longer be true.
\end{rem}

The techniques employed in the proof of Theorem \ref{thm: dimB Sp} are in fact very similar to those needed to determine the Assouad spectrum of $S_p$. By adjusting the right quantitative estimates, we can modify accordingly the arguments within the previous proof, so that they apply in the context of two small scales, $r$ and $r^\theta$, as in the definition of the $\theta$-Assouad spectrum.

\begin{thm}\label{thm: A-spec of S_p}
    Let $S_0\subset \R^d$ be a topological sphere and $S_p=\bigcup_{n\in \N_0} S_n$ be a $p$-concentric sphere collection generated by $S_0$ as in Theorem \ref{thm: dimB Sp}, with the additional assumption that $\dim_A S_0=d-1$\footnote{Requiring $\dim_B S_0=\dim_A S_0$ might seem restrictive at first. However, various of the lines of research outlined in the Introduction focus on  \textit{$d_0$-Ahlfors regular} spheres (see \cite{Bonk_QC, Bonk_Meyer} for definition and details), which is a special case of the spheres with $\dim_B S_0=\dim_A S_0$. { Hence, it is natural to consider $\dim_B S_0=\dim_A S_0=d-1$ in this case.}}. 
    If $p (d-1)\leq 1$, then  we have
    \begin{equation}\label{eq: main dim_A pd<1}
    \dim_A^\theta S_p = \min \left\{ \frac{d}{(1+p)(1-\theta)}, \,d \right\},
    \end{equation} for all $\theta\in (0,1)$. On the other hand, if $p (d-1)>1$, then
    \begin{equation}\label{eq: main dim_A pd>1}
        \dim_A^\theta S_p = \min \left\{ d-1+ \frac{\theta}{p(1-\theta)},\, d\right\},
    \end{equation} for all $\theta\in (0,1)$.
\end{thm}

\begin{proof}
Recall that $d_0=\dim_B S_0=d-1=\dim_A S_0$. For any $\theta\geq p/(p+1)$, both \eqref{eq: main dim_A pd<1} and \eqref{eq: main dim_A pd>1} result in $\dim_A^\theta S_p=d$. Hence, by continuity of the Assouad spectrum as a function of $\theta$, it is enough to show \eqref{eq: main dim_A pd<1} and \eqref{eq: main dim_A pd>1} for all $\theta<p/(p+1)$.

Fix $\theta\in (0,p/(p+1))$ and positive $d_1, d_2$ close to $d_0$, such that $d_1<d_0<d_2$. Moreover, fix $r\in (0,1)$, $\epsilon>0$ as small as necessary for the following arguments, in a way that depends only on  uniform constants of $S_p$ and $\R^d$, similarly to the proof of Theorem \ref{thm: dimB Sp}. Denote by $n_r\in \N$ the unique positive integer satisfying
\begin{equation}\label{eq: dimA choice of n_r}
n_r^{-p}-(n_r+1)^{-p} \leq \frac{r}{c_2d^{1/2}} < (n_r-1)^{-p}-n_r^{-p},
\end{equation} and by $m_r\in \N$ the unique positive integer satisfying
\begin{equation}\label{eq: dimA choice of m_r theta}
m_r^{-p} \leq \frac{r^\theta}{c_1} < (m_r-1)^{-p}.
\end{equation} Since $r$ was fixed to be small enough, the integers $n_r$, $m_r$ are large enough, so that by \eqref{eq: dimA choice of n_r} we have
\begin{equation}\label{eq: dimA nr and r compar}
    n_r\simeq r^{-\frac{1}{p+1}},
\end{equation} and by \eqref{eq: dimA choice of m_r theta} we have
\begin{equation}\label{eq: dimA mr and r-theta compar}
    m_r\simeq r^{-\frac{\theta}{p}}.
\end{equation} Due to the choice $\theta<p/(p+1)$, by \eqref{eq: dimA choice of n_r} and \eqref{eq: dimA choice of m_r theta} we also have that $m_r\leq n_r$. Note that the Assouad spectrum is stable under taking the closure of a set (see \cite[p.18, 49]{Fra_book}), which ensures that we may estimate the Assouad spectrum of the closure of $S_p$ instead. Therefore, henceforth we follow the convention that the center $0$ of $S_p$ is also part of the concentric sphere collection. Moreover, despite the fact that the definition of the Assouad spectrum requires to check balls centered at all points of the given set, we focus on $B(0,r^\theta)$ instead, for reasons that become clear in later.

Arguing similarly to the proof of Theorem \ref{thm: dimB Sp} and splitting $B(0,r^\theta)\cap S_p$ into two disjoint sets, we have
$$
N_r(B(0,r^\theta)\cap S_p)\simeq N_r(g_{n_r}(B_0))+\sum_{j=m_r}^{n_r} N_r(S_j),
$$ which by \eqref{eq: N_r similar} and the choices of $\epsilon, d_1, d_2$ implies that
$$
\sum_{j=m_r}^{n_r} \frac{j^{-pd_1}}{r^{d_1}}\lesssim N_r(B(0,r^\theta)\cap S_p)\lesssim \left(\frac{n_r^{-p}}{r}\right)^{d+\epsilon}+\sum_{j=m_r}^{n_r} \frac{j^{-pd_2}}{r^{d_2}}.
$$ Using \eqref{eq: dimA nr and r compar} and standard partial summation techniques on the above relation results in \small
\begin{equation}\label{eq: dimA Nr double rel}
    r^{-d_1}\max\{ m_r^{1-pd_1}, n_r^{1-pd_1} \}\lesssim N_r(B(0,r^\theta)\cap S_p)\lesssim r^{\left(\frac{p}{p+1}-1\right)(d+\epsilon)}+r^{-d_2}\max\{ m_r^{1-pd_2}, n_r^{1-pd_2}\}.
\end{equation} \normalsize The rest of the proof is a case study on the product $pd_0$, similarly to the proof of Theorem \ref{thm: dimB Sp}, while also accounting for the different values of $\theta$ that could potentially determine which power of $r$ is dominant.

Suppose $p d_0<1$. In this case, we only need to find a sequence of balls centered at $S_p$ and bound the corresponding covering number from below, which yields a lower bound on $\dim_A^\theta S_p$. This is because the desired upper bound on $\dim_A^\theta S_p$ follows trivially by Theorem \ref{thm: dimB Sp} and Lemma \ref{le: dimA upper bd}. Hence, we can focus on balls centered at $0$ and the lower bound of \eqref{eq: dimA Nr double rel}. Namely, by \eqref{eq: dimA Nr double rel} and $p d_1<1$ we have
$$
r^{-d_1}n_r^{1-pd_1}\lesssim N_r(B(0,r^\theta)\cap S_p),
$$ which by \eqref{eq: dimA nr and r compar} implies
$$
N_r(B(0,r^\theta)\cap S_p)\gtrsim r^{-d_1-\frac{1-pd_1}{p+1}}= (r^{\theta-1})^\frac{-d_1(p+1)-(1-pd_1)}{(\theta-1)(p+1)}=(r^{\theta-1})^\frac{d_1+1}{(1-\theta)(p+1)}.
$$ Thus, the above shows that
$$
\dim_A^\theta S_p\geq \frac{d_1+1}{(1-\theta)(p+1)}, 
$$ for arbitrary $d_1$ close to $d_0=d-1$, which is enough to complete the proof of \eqref{eq: main dim_A pd<1}.

Suppose that $p d_0>1$. Note at this point that if $x\in S_p$ with $|x|\leq 2 r^\theta$, then
\begin{equation}\label{eq: dimA Nx less N0}
    N_r(B(x,r^\theta)\cap S_p)\leq N_r(B(0,10r^\theta)\cap S_p)\lesssim N_r(B(0,r^\theta)\cap S_p),
\end{equation} where the comparability constant only depends on the ambient space $\R^d$ (see for instance \cite[Section 13.1]{Fra_book}). Hence, it remains to study the covering number $N_r(B(0,r^\theta)\cap S_p)$, and $N_r(B(x,r^\theta)\cap S_p)$ for arbitrary $x\in S_p$ with $|x|>2 r^\theta$. In the latter case, we observe that the individual spheres $S_j$ that intersect $B(x,r^\theta)$ are not more than $m_r$ in number, due to \eqref{eq: dist(S_n,0)} and \eqref{eq: dimA choice of m_r theta} (see also \cite[p.~3263]{Fra_spirals} for a similar argument). Therefore, using the fact that $\dim_A S_0=d_0=\dim_B S_0$, which implies that $\dim_A^\theta S_0= d_0$ by \eqref{eq: spec bounds}, and a similarity argument resembling \eqref{eq: N_r similar}, we have 
$$
N_r(B(x,r^\theta)\cap S_p)\lesssim m_r (r^{\theta-1})^{d_2}.
$$ By \eqref{eq: dimA mr and r-theta compar}, the above implies that
\begin{equation}\label{eq: dimA Nr(x)}
    N_r(B(x,r^\theta)\cap S_p)\lesssim r^{-\frac{\theta}{p} }(r^{\theta-1})^{d_2}= (r^{\theta-1})^{d_2+\frac{\theta}{p(1-\theta)}}.
\end{equation}
Hence, due to \eqref{eq: dimA Nx less N0} and the above inequality, it is enough to focus on achieving lower and upper bounds for $N_r(B(0,r^\theta)$ for the rest of the proof. The relation $p d_0>1$ implies for \eqref{eq: dimA Nr double rel} that
$$
r^{-d_1}m_r^{1-pd_1}\lesssim N_r(B(0,r^\theta)\cap S_p)\lesssim r^{\left(\frac{p}{p+1}-1\right)(d+\epsilon)}+r^{-d_2} m_r^{1-pd_2}.
$$ Using \eqref{eq: dimA mr and r-theta compar} in the above relation results in
$$
    r^{-d_1}r^{-\frac{\theta(1-pd_1)}{p}}\lesssim N_r(B(0,r^\theta)\cap S_p)\lesssim r^{\frac{-(d+\epsilon)}{p+1}}+r^{-d_2} r^{-\frac{\theta(1-pd_2)}{p}}.
$$ Note that
$$
-\frac{d_1}{\theta-1}-\frac{\theta(1-pd_1)}{p(\theta-1)}= d_1+\frac{\theta}{p(1-\theta)},
$$ and similarly for $d_2$ instead of $d_1$, which when applied to the previous inequality yield
\begin{equation}\label{eq: dimA Nr rel for pd0>1}
    (r^{\theta-1})^{d_1+\frac{\theta}{p(1-\theta)}}\lesssim N_r(B(0,r^\theta)\cap S_p)\lesssim (r^{\theta-1})^{\frac{d+\epsilon}{(p+1)(1-\theta)}}+(r^{\theta-1})^{d_2+\frac{\theta}{p(1-\theta)}}.
\end{equation} Regarding the dominant term in the upper bound of \eqref{eq: dimA Nr rel for pd0>1}, note that
$$
d_0+\frac{\theta}{p(1-\theta)}>\frac{d}{(p+1)(1-\theta)}
$$ if, and only if
$$
\theta< \frac{p(pd_0-d_0-d)}{(p+1)(pd_0-1)}=\frac{p}{p+1}.
$$ Hence, by choice of $\theta$, and by taking $\epsilon$ to be small enough, we have 
$$
d_2+\frac{\theta}{p(1-\theta)}>\frac{d+\epsilon}{(p+1)(1-\theta)},
$$ which if applied to \eqref{eq: dimA Nr rel for pd0>1} yields
$$
(r^{\theta-1})^{d_1+\frac{\theta}{p(1-\theta)}}\lesssim N_r(B(0,r^\theta)\cap S_p)\lesssim (r^{\theta-1})^{d_2+\frac{\theta}{p(1-\theta)}}.
$$ As a result, by the above and by \eqref{eq: dimA Nx less N0}, \eqref{eq: dimA Nr(x)}, we have shown that
$$
d_1+\frac{\theta}{p(1-\theta)}\leq \dim_A^\theta S_p\leq d_2+\frac{\theta}{p(1-\theta)},
$$ for $d_1$ and $d_2$ arbitrarily close to $d_0$, which is enough to complete the proof in this case as well.

The case $p d_0=1$ is treated similarly to the above cases, so we omit the details. Notice that in this case we have $pd_1<1$ and $1<pd_2$, which yield different dominant terms on the two sides of \eqref{eq: dimA Nr double rel}, but the arguments are identical for each side to the corresponding cases; namely, the lower bound is treated using $pd_1<1$ as in the case $pd_0<1$, and the upper bound is treated using $pd_2>1$ as in the case $pd_0>1$. This completes the proof.
\end{proof}

\begin{rem}
We emphasize that in the case $p(d-1)>1$, the Assouad spectrum is not given by the upper bound in Lemma \ref{le: dimA upper bd} for $0<\theta<\frac{p}{1+p}$. This provides us with a family of such examples in every Euclidean space $\R^d$.

\end{rem}

Note that for $p>(d-1)^{-1}$, the box dimension of $S_p$ is trivial, in the sense that $\dim_B S_p=d-1$. On the other hand, the Assouad spectrum provides non-trivial information on $S_p$, even in the case where the box dimension fails to do so. 
Moreover, \eqref{eq: main dim_A pd<1}, \eqref{eq: main dim_A pd>1}  yield that the Assouad dimension of $S_p$ is equal to $d$. { Hence, $\dim_A S_p$ also provides trivial information when compared to the Assouad spectrum.} The following is an immediate implication of \eqref{eq: spec bounds}. 
\begin{cor} 
For all $p>0$, we have  $\dim_A S_p = \qad S_p = d$, where $S_p$ is as in Theorem \ref{thm: A-spec of S_p}. 
\end{cor}

We finish this section by showing that a topological sphere of large {  lower} box dimension cannot generate a  polynomially concentric sphere collection.

\begin{proof}[Proof of Theorem \ref{thm: no p-pol concentric sphere coll}]
Assume towards contradiction that $S_p$ is a $p$-polynomially concentric sphere collection, as described in the beginning of Section \ref{sec: dimensions}, but with { $d_0=\lodimB S_0>d-1$}. Note that we do not require $S_0$ to satisfy $\dim_A S_0 =d_0$. We arrive at a contradiction by providing an absurd lower bound on the Assouad spectrum of $S_p$ for certain $\theta\in (0,1)$.

Fix $d_1<d_0$ close to $d_0$, such that $d_1>d-1$, and $\theta\in (0,\frac{p}{p+1})$. Moreover, similarly to the proof of Theorem \ref{thm: A-spec of S_p}, fix small $r\in (0,1)$ and integers  $n_r$, $m_r$ so that \eqref{eq: dimA choice of n_r}, \eqref{eq: dimA choice of m_r theta} are satisfied. Choosing $r$ small enough ensures that $n_r\simeq r^{-\frac{1}{p+1}}$ and $m_r\simeq r^{-\frac{\theta}{p}}$, which by choice of $\theta<p/(p+1)$ also ensures that $m_r\leq n_r$. {  Arguing exactly as in the proof of Theorem \ref{thm: A-spec of S_p}, and by using the left-hand side of \eqref{eq: N_r rel by dimB} and \eqref{eq: N_r similar}, yields}
\begin{equation}\label{eq: imposs Nr}
    N_r(B(0,r^\theta)\cap S_p)\gtrsim \sum_{j=m_r}^{n_r} \frac{j^{-pd_1}}{r^{d_1}}\gtrsim r^{-d_1}\max\{ m_r^{1-pd_1}, n_r^{1-pd_1} \}.
\end{equation} 

Suppose $pd_0>1$. Then $d_1$ can also be chosen so that $p d_1>1$, which implies by \eqref{eq: imposs Nr} that
$$
N_r(B(0,r^\theta)\cap S_p)\gtrsim (r^{\theta-1})^{d_1+\frac{\theta}{p(1-\theta)}},
$$ exactly as in the proof of Theorem \ref{thm: A-spec of S_p}. Since $r$, $\theta$ are arbitrary, this implies that
$$
\dim_A^\theta S_p \geq d_1+\frac{\theta}{p(1-\theta)},
$$ which also yields
\begin{equation}\label{eq: imposs dimA cont 1}
    \dim_A^\theta S_p \geq d_0+\frac{\theta}{p(1-\theta)},
\end{equation}
because $d_1<d_0$ can be arbitrarily close to $d_0$. However, note that $d_0+\frac{\theta}{p(1-\theta)}<d$ if, and only if, $\theta<\frac{p(d-d_0)}{p(d-d_0)+1}$. Since $\frac{p(d-d_0)}{p(d-d_0)+1}<\frac{p}{p+1}$, due to $d_0>d-1$, if we pick
$$
\frac{p(d-d_0)}{p(d-d_0)+1}<\theta<\frac{p}{p+1},
$$ then \eqref{eq: imposs dimA cont 1} implies that $\dim_A^\theta S_p> d,$ which is a contradiction.

On the other hand, if $pd_0\leq 1$, we have $pd_1<1$, which implies by \eqref{eq: imposs Nr} that
$$
N_r(B(0,r^\theta)\cap S_p)\gtrsim (r^{\theta-1})^{\frac{d_1+1}{(1-\theta)(p+1)}}.
$$ This similarly yields
$$
\dim_A^\theta S_p \geq \frac{d_0+1}{(1-\theta)(p+1)},
$$ which by choosing
$$
\frac{d(p+1)-(d_0+1)}{d(p+1)}<\theta<\frac{p}{p+1},
$$ provides the contradicting estimate $\dim_A^\theta S_p>d$ for this case as well, completing the proof.
\end{proof}

\begin{rem}
    For $p<d^{-1}(1+d_0-d)$, if we let $d_1<d_0$ arbitrarily close to $d_0$ and follow the argument of the proof of Theorem \ref{thm: dimB Sp} just for the lower bound of $N_r(S_p)$, we get that $\ovdimB S_p\geq\frac{d_0+1}{1+p}>d$, by choice of $p$, which is a contradiction. Hence, the upper box dimension of $S_p$ can also trace the impossibility of such a concentric collection, but only for certain values of $p$. Therefore, the investigation of the Assouad spectrum is truly necessary.
\end{rem}

\section{Quasiconformal classification of spiral shells} \label{sec: QC-Holder parametrization}

As mentioned in the Introduction, quasiconformal maps are a natural generalization of conformal maps in higher dimensions. There are multiple ways to express the intuition behind quasiconformality, i.e., the property that infinitesimal balls are mapped onto infinitesimal ellipsoids whose eccentricity is globally bounded by a constant. We state the analytic definition of quasiconformal maps, which relates directly to the dimension distortion result of interest (see \cite{Vais_book} for a thorough exposition). A homeomorphism $f:\Omega \to \Omega'$ between domains in $\R^d$, $d \geq 2$, is said to be {\it $K$-quasiconformal}, for some $K\geq 1$, if $f$ lies in the local Sobolev space $W^{1,d}_{\text{loc}}(\Omega:\R^d)$, and the inequality
$$
|Df|^d \le K \det Df
$$
holds almost everywhere in $\Omega$ (with respect to the $d$-Lebesgue measure). Note that $Df$ denotes the (a.e. defined) differential matrix, and $|\mathbf A|=\max\{|\mathbf A(\mathbf v)|:|\mathbf v|=1\}$ denotes the operator norm of a matrix $\mathbf A$. Recall that the local Sobolev space $W^{1,d}_{\text{loc}}(\Omega:\R^d)$ is the space of maps defined in $\Omega$ with first-order weak derivatives in $L^p_{\text{loc}}(\Omega:\R^d)$ (see \cite{Heinonen} for more details).

There are many dimension distortion results, for different analytic and geometric mapping classes related to quasiconformal maps, and regarding different dimension notions. See  \cite{GehringVais}, \cite{Kaufman}, \cite{BaloghAGMS}, \cite{btw:heisenberg}, \cite{OurQCspec}, \cite{SobChron}, \cite{HolomSpecChron} for a non-exhaustive list. In fact, the following theorem proved by Tyson and the author (see \cite[Theorem 1.3]{OurQCspec}) is crucial for the proof of Theorem \ref{thm: main QC shells}.

\begin{thm}\label{thm: ChrTys}
Let $F:\Omega\to\Omega'$ be a $K$-quasiconformal map between domains $\Omega, \,\Omega'\subset \R^d$. For $t>0$, set $\theta(t)=1/(t+1)$. If $E$ is a non-empty bounded set  with closure $\overline{E}\subset\Omega$, and $\alpha_t=\dim_{A,reg}^{\theta(t)}E$, then
\begin{equation}\label{eq: QC spec}
\left( 1 - \frac{d}{s} \right) \left( \frac1{\alpha_{t/K}} - \frac1d \right) 
 \le \frac1{\dim_{A,reg}^{\theta(t)}F(E)} - \frac1d \le 
\left( 1 - \frac{d}{s} \right)^{-1} \left( \frac1{\alpha_{Kt}} - \frac1d \right),
\end{equation} for a fixed $s=s(d,K)>d$ that only depends on the space $\R^d$ and the dilatation constant $K$ of the map $F$.
\end{thm}

We have all the tools necessary to fully classify the polynomial spiral shells up to quasiconformal equivalence.
\begin{proof}[Proof of Theorem \ref{thm: main QC shells}]
One direction is immediate, as there exists a  map $F$ with $F(\Sh_p)=\Sh_q$, which is $K$-quasiconformal for any $K\geq p/q$. Namely, the radial stretch map $F:\R^d\rightarrow \R^d$ with $F(0)=0$ and
$$
F(x) = |x|^{q/p-1}x,
$$ for all $x\in \R^d\setminus\{0\}$, is $(p/q)$-quasiconformal (see \cite[Section 16.2]{Vais_book}), which implies that $F$ is $K$-quasiconformal for all $K\geq p/q$. Moreover, it is elementary to show that $F$ maps $\Sh_p$ onto $\Sh_q$, as required. 

For the other implication, assume towards contradiction that $F:\R^d\rightarrow \R^d$ is a $K$-quasiconformal map with $K<p/q$ and $F(\Sh_p)=\Sh_q$. We aim to calculate the regularized Assouad spectrum of the shells $\Sh_p, \Sh_q$, and use Theorem \ref{thm: ChrTys} to arrive to a contradiction. Set
$S_p:= \cup_{j=1}^\infty S(0,j^{-p})$ and $S_q:= \cup_{j=1}^\infty S(0,j^{-q})$. The sets $S_p$, $S_q$ are polynomially concentric sphere collections centered at $0$ and generated by the unit sphere $S(0,1)$ in $\R^3$. Hence, by Theorem \ref{thm: A-spec of S_p} we have that
\begin{equation}\label{eq: dimASp QCpf}
    \dim_A^\theta S_p =\left\{
\begin{array}{ll}
      \frac{3}{(1+p)(1-\theta)},\, & \text{if} \,\,\,0<\theta <\frac{p}{p+1}, \,\, p\leq \frac{1}{2} \\
      2+\frac{\theta}{p(1-\theta)},\, & \text{if} \,\,\, 0<\theta <\frac{p}{p+1}, \,\, p> \frac{1}{2}\\
      3,\, & \text{if} \,\,\, \frac{p}{p+1}\leq\theta<1, \\
\end{array} 
\right. 
\end{equation} and similarly for $\dim_A^\theta S_q$.

Note that a bi-Lipschitz map $h_1$ can be constructed in order to map $\Sh_p^{\text{u}}:=\Sh_p\cap \{ 
(x_1,x_2,x_3): \, x_3\geq 0 \}$ onto $S_p^{\text{u}}:= S_p\cap \{ 
(x_1,x_2,x_3): \, x_3\geq 0 \}$ (see \cite{Falc_spirals, Fra_spirals} for similar arguments). Similarly, a different bi-Lipschitz map $h_2$ can be constructed in order to map $\Sh_p^{\ell}:=\Sh_p\cap \{ 
(x_1,x_2,x_3): \, x_3\leq 0 \}$ onto $S_p^{\ell}:= S_p\cap \{ 
(x_1,x_2,x_3): \, x_3\leq 0 \}$. Using the finite and bi-Lipschitz stability of the Assouad spectrum (see \cite[p.~18, 49]{Fra_book}, we have
$$
\dim_A^\theta \Sh_p = \max\{\dim_A^\theta \Sh_p^{\text{u}}, \dim_A^\theta \Sh_p^\ell\}= \max\{\dim_A^\theta S_p^{\text{u}}, \dim_A^\theta S_p^\ell\}= \dim_A^\theta S_p,
$$ for all $\theta\in (0,1)$. Similarly we can show that $\dim_A^\theta \Sh_q= \dim_A^\theta S_q$. Hence, by the above and \eqref{eq: dimASp QCpf}, we can fully determine the Assouad spectrum of $\Sh_p$ and $\Sh_q$. Moreover, since the expressions in \eqref{eq: dimASp QCpf} are monotone, by the relation of the Assouad spectrum with the regularized Assouad spectrum in \eqref{eq:regularization}, we have

\begin{equation}\label{eq: dimAregSp QCpf}
    \dim_{A,reg}^\theta \Sh_p =\left\{
\begin{array}{ll}
      \frac{3}{(1+p)(1-\theta)},\, & \text{if} \,\,\,0<\theta <\frac{p}{p+1}, \,\, p\leq \frac{1}{2} \\
      2+\frac{\theta}{p(1-\theta)},\, & \text{if} \,\,\, 0<\theta <\frac{p}{p+1}, \,\, p> \frac{1}{2}\\
      3,\, & \text{if} \,\,\, \frac{p}{p+1}\leq\theta<1, \\
\end{array} 
\right. 
\end{equation} and similarly for $\dim_{A,reg}^\theta \Sh_q$.

We have all the necessary dimension-theoretic information in order to apply  Theorem \ref{thm: ChrTys}. In fact, we mostly use the relation between the phase transitions $\theta_p=p/(p+1)$ and $\theta_q=q/(q+1)$ of the two shells, i.e., the values where the regularized Assouad spectrum reaches the (quasi-)Assouad dimension. Set $t=1/q$. Then $\theta(t) = 1/(1+t) = q/(1+q)$, and so $\dim_A^{\theta(t)}(S_q) = 3$ by \eqref{eq: dimAregSp QCpf}. On the other hand,
$$
\theta(t/K) = \frac{K}{K+t}= \frac{q K}{qK+1} < \frac{p}{p+1},
$$ 
due to the fact that $K<p/q$. Therefore, $\dim_A^{\theta(t/K)}(S_p) < 3$. This leads to a contradiction due to Theorem \ref{thm: ChrTys} and, in particular, due to the left-hand side of \eqref{eq: QC spec}. Therefore, such a $K$-quasiconformal map does not exist.
\end{proof}

\begin{rem}
    A closer analysis of the above proof indicates that Theorem \ref{thm: main QC shells} could be generalized to non-global quasiconformal maps. More specifically, the result is still true if the map $F:\R^d\rightarrow \R^d$ in the statement of the theorem is replaced by a map $F:\Omega\rightarrow\Omega'$ between domains of $\R^d$ such that $\Sh_p\cup \{0\}\subset \Omega$ and $\Sh_q\cup \{0\}\subset \Omega'$.
\end{rem}

We emphasize that the Assouad spectrum estimates are crucial in the proof of Theorem \ref{thm: main QC shells}. Indeed, in the case $1/2\leq q\leq p$, all other dimension notions do not provide non-trivial information, since $\dim_H \Sh_p=\dim_B \Sh_p=2<3=\dim_A \Sh_p$, where $\dim_H$ denotes the Hausdorff dimension (see \cite{falconer}), and similarly for $\Sh_q$. The available dimension distortion results for the Hausdorff, box, and Assouad dimension under quasiconformal maps are not enough to classify such polynomial shells. See \cite{OurQCspec} for a more extensive discussion on this phenomenon in the case of polynomial spirals.

\section{Final Remarks}\label{sec: Final Remarks}

Suppose $S_0=f(S(0,1))\subset \R^d$ is a topological sphere, where $f:\R^d\rightarrow \R^d$ is a homeomorphism, and $x_0$ is a point lying in the interior of the topological ball $f(B(0,1))$. We address the existence of a concentric sphere collection generated by $S_0$, regardless of the value of $\dim_B S_0\in [d-1,d]$. Since $f(B(0,1))$ is an open set, there is a maximal open ball lying in $f(B(0,1))$, centered at $x_0$ and of radius $r_1>0$. Set $g_1$ to be a similarity, say of ratio $r_1/100$, such that $x_0\in g_1(B(0,1))\subset B(x_0,r_1)$. Repeating the same argument for $g_1(B(0,1))$ and iterating infinitely many times, we construct a sequence of similarities $\{g_n\}_{n\in \N}$ that contracts $S_0$ into the point $x_0$, i.e., a concentric sphere collection centered at $x_0$ and generated by $S_0$. However,  by Theorem \ref{thm: no p-pol concentric sphere coll}, we have established that the similarity ratios of $g_n$ cannot be comparable to any sequence of the form $n^{-p}$, for any $p>0$, {  unless $\lodimB S_0=d-1$. Thus, in the case $\lodimB S_0\in (d-1,d]$,} the sequence $a_n$ in the constructed example has to converge to $0$ faster than $n^{-p}$ does. Note that in such situations, the convergence is typically too fast for any dimension notion to trace  ``fractalness" in the collection. For instance, following the method of the proof of Theorem \ref{thm: A-spec of S_p}, it can be shown that a concentric sphere collection with sequence $a_n=e^{-n}$ has Assouad spectrum (and, similarly, Assouad dimension) equal to $\dim_B S_0$. 
A natural question motivated by this discussion is to try and fully characterize the topological spheres $S_0$ from which a polynomially concentric collection can be generated. Given a topological sphere $S_0\subset \R^d$, Theorem \ref{thm: no p-pol concentric sphere coll} shows that $\lodimB S_0=d-1$ is a necessary condition. If $\lodimB S_0=d-1$, under what conditions can we polynomially contract $S_0$ into a point using similarities? Are there specific restrictions imposed on the degree $p$ of the polynomial rate $n^{-p}$, based on properties of $S_0$?

It should be emphasised that the case $p>(d-1)^{-1}$ in Theorem \ref{thm: A-spec of S_p} is another one of the few examples of sets $E\subset \R^d$, where $\dim_A^\theta E$ is not constant and is not equal to $\dim_B E/(1-\theta)$, for $\theta<1-\dim_B E/d$ (see Lemma \ref{le: dimA upper bd}). Many of the first such examples mostly appeared in $\R^2$, such as the polynomial spirals studied by Fraser in \cite{Fra_spirals}, and the elliptical polynomial spirals studied by Falconer-Fraser-Burrell in \cite{Falc_spirals} . In fact, the examples studied in the latter paper stand out for one more reason; for fixed $p,q$, the Assouad spectrum of elliptical spirals $\Sh_{p,q}$ exhibits two distinct phase transitions, i.e., $\dim_A^\theta \Sh_{p,q}$ has three different branches as a function of $\theta$, unlike $\dim_A^\theta \Sh_p$, which only has two (see \cite[Theorem 2.6]{Falc_spirals}). This motivates the study of another class in $\R^d$, namely that of \textit{concentric ellipsoid collections}. In particular, this class is defined similarly to the concentric sphere collections in Definition \ref{def: concentric}, but with the second condition (ii) replaced by $\dist(S_n, S_{n+1})= c_2 b_n$, for some strictly decreasing sequence $\{b_n\}_{n\in\N}$ that converges to $0$ at a different rate than $\{a_n\}_{n\in\N}$. We expect the dimension study of these objects to be quite interesting, potentially providing  more examples of sets with Assouad spectrum of two distinct phase transition values.

Recall that a homeomorphism $f:\R^d\rightarrow \R^d$ is $(\alpha,\beta)$-\textit{bi-H\"older continuous} on a non-empty set $E$, with H\"older exponents $0<\alpha\leq 1\leq \beta<\infty$, if there is a constant $C>0$ such that
$$
C^{-1}|x-y|^\beta\leq |f(x)-f(y)|\leq C |x-y|^\alpha,
$$ for all $x,y\in E$.
Fraser in \cite{Fra_spirals} suggested the following classification-like direction: given two bounded sets $X,Y\subset \R^d$, try to construct a bi-H\"older map $f$ such that $f(X)=Y$. Moreover, determine what is the sharpest possible choice of H\"older exponents for such a map (which depends on properties of $X, Y$), and compare them to the bounds on the exponents from dimension distortion theorems. We have indirectly stepped towards this direction for spiral shells through Theorem \ref{thm: main QC shells}, and the connection between quasiconformal and bi-H\"older maps.
Namely,  any $K$-quasiconformal map $F:\R^d\rightarrow\R^d$ is in fact $(1/K,K)$-bi-H\"older continuous on every compact set $E\subset \R^d$ (see for instance \cite[Theorem 7.7.1]{Iwaniec}). Hence, in the proof of Theorem \ref{thm: main QC shells}, we provided a $(q/p,p/q)$-bi-H\"older map that maps $\Sh_p\cup\{0\}$ onto $\Sh_q\cup\{0\}$. However, while this map is sharp for the quasiconformal classification of spiral shells, it is not clear whether it is sharp for the bi-H\"older class of maps between such shells. 

One could use the bi-H\"older dimension distortion results for the Assouad spectrum (see \cite[Section 3.4.3]{Fra_book}), and try to come up with bounds on the H\"older exponents of a map that sends $\Sh_p$ onto $\Sh_q$, in terms of $p$ and $q$. For instance, if $f:\R^d\rightarrow\R^d$ is $(\alpha,\beta)-$bi-H\"older and $E\subset \R^d$ is bounded,  then  
$$
\frac{\dim_B E}{\beta}\leq \dim_B f(E)\leq \frac{\dim_B E}{\alpha}.
$$
Therefore, in the case $0<q<p<1$, using the the box dimension of $\Sh_p, \Sh_q$ and the above inequality, we arrive at the trivial bound $\beta\geq (q+1)/(p+1)$, and the bound $\alpha\leq (q+1)/(p+1)$ for the H\"older exponents. This does not necessarily imply that a $\frac{q+1}{p+1}$-H\"older map $f$ with $f(\Sh_p)=\Sh_q$ actually exists. Note that for $0<q<p$ we have
$$
\frac{q}{p}<\frac{q+1}{p+1},
$$ so the latter term is indeed a ``better" H\"older exponent, in the sense that it leads to a ``smoother" map $f$.
However, it appears that apart from certain specialized examples, the H\"older dimension distortion estimates are rarely sharp. This was also emphasized in the H\"older unwinding problem for polynomial spirals by Fraser in \cite{Fra_spirals}. As a result, it is not unreasonable to suspect that $q/p$ is the sharpest exponent possible, at least for the classification problem involving $(\alpha,\alpha^{-1})$-bi-H\"older maps between spiral shells $\Sh_p,\Sh_q$. In fact, a natural direction to explore is the following; suppose the sharp bi-H\"older exponent bounds in a classification problem are not achieved through knowledge of the dimensions of the sets and the H\"older dimension distortion estimates, are the bounds achieved by the corresponding quasiconformal classification results  sharp?

\end{document}